\documentclass[10pt, reqno]{amsart}
\usepackage{graphicx}
\usepackage{xcolor}
 \usepackage{soul}
 \usepackage{ulem}
\def\constr#1^#2{\mathrel{\mathop{\kern 0pt#1}\limits^{#2}}}
\def\build#1_#2{\mathrel{\mathop{\kern 0pt#1}\limits_{#2}}}

\theoremstyle{plain}

\newtheorem{theorem}{Theorem}[section]

\newtheorem{lemma}{Lemma}[section]

\theoremstyle{definition}
\newtheorem{definition}{Definition}[section]
\theoremstyle{remark}
\newtheorem{remark}{\bf Remark}[section]

\theoremstyle{remark}
\newtheorem{remarks}{\bf Remarks}[section]
\theoremstyle{remark}
\newtheorem{exam}{\bf Example}[section]

\numberwithin{equation}{section}

\makeatletter
\@namedef{subjclassname@2020}{%
  \textup{2020} Mathematics Subject Classification}
\makeatother

\newcommand{\nsubseteq}{\nsubseteq}

\newcommand{\ds}{\displaystyle}

\newcommand{\be}{\begin{equation}}
\newcommand{\ee}{\end{equation}}

\DeclareGraphicsExtensions{.eps} \DeclareGraphicsExtensions{.jpeg}

\begin{document}
\title[Delayed abstract thermoelastic systems]
{Lack of differentiability of semigroups associated to delayed abstract thermoelastic systems}
\thispagestyle{empty}

\author{Ka\"{\i}s Ammari}
\address{LR Analysis and Control of PDEs, LR 22ES03,
Department of Mathematics, Faculty of Sciences of Mona-
stir, University of Monastir, 5019 Monastir, Tunisia}
\email{kais.ammari@fsm.rnu.tn}

\author{Makrem Salhi}
\address{
LR Analysis and Control of PDEs, LR 22ES03 (FSM-UM) $\&$ Department of Mathematics and Computer Science,
Preparatory Institute for Engineering Studies of Sfax, University of Sfax, Tunisia}
\email{makram.salhi@ipeis.usf.tn}

\author{Farhat Shel}
\address{
LR Analysis and Control of PDEs, LR 22ES03 (FSM-UM) $\&$ Department of Mathematics and Computer Science,
Preparatory Institute for Engineering Studies of Sfax, University of Sfax, Tunisia}
\email{farhat.shel@fsm.rnu.tn}

\date{}

\begin{abstract} In this paper, we consider the associated semigroups to some abstract thermoelastic systems (in particular the $\alpha$-$\beta$ system), with a partial delay on the coupled system. We will prove that the corresponding semigroups (in appropriate Hilbert spaces) are not differentiable.
\end{abstract}


\maketitle

\tableofcontents

\section{Introduction}

Let $H$ be a Hilbert space equipped with an inner product
$(.,.)_{H}$ and an induced norm $\|\;\|_{H}$.
We consider the following two abstract thermoelastic systems with delay. In the first, the delay appear at the hyperbolic part:
 \begin{equation}\label{s0}
\left\{
\begin{array}{ll}
u''(t)+ Au(t-\tau)+aAu'(t)-A^{\beta}\theta(t)=0, & \quad t\in(0,+\infty), \\
\theta'(t)+A^{\alpha}\theta(t)+A^{\beta}u'(t)=0, & \quad t\in(0,+\infty), \\
u(0)=u_{0}, u'(0)=u_{1}, \theta(0)=\theta_{0}, &  \\
\ds A^{\frac{1}{2}}u(t-\tau)=\phi(t-\tau), & \quad t\in(0,\tau),
\end{array}
\right.
\end{equation}
and in the second, the delay appear at the parabolic part:
\begin{equation}\label{s1}
\left\{
\begin{array}{ll}
u''(t)+Au(t)-A^{\beta}\theta(t)=0, & \quad t\in(0,+\infty), \\
\theta'(t)+\kappa A^\alpha\theta(t-\tau)+aA^\alpha\theta(t)
+A^{\beta}u'(t)=0, & \quad t\in(0,+\infty), \\
u(0)=u_{0}, u'(0)=u_{1}, \theta(0)=\theta_{0}, &  \\
\ds A^{\frac{\alpha}{2}} 
\theta (t-\tau) = \psi (t-\tau), & \quad t \in (0,\tau).
\end{array}
\right.
\end{equation}
where $A:D(A)\subset H\rightarrow H$ is a self-adjoint, positive definite (unbounded) operator; $\tau, a, \kappa$ are positive real numbers and $(\alpha,\beta)$ in the region $Q$ defined as follows
$$\fbox{$ Q:=\{(\alpha,\beta)\in[0,1]\times [0,1]\mid 2\beta-\alpha\leq 1\}$}.$$

We suppose that $A$  admits a sequence of eigenvalues $\mu_n$ such that
$$\lim_{n\rightarrow \infty} \mu_n =\infty.$$

In  \cite{ASS24} we introduced the auxiliary variable $z$ defined by
$$\ds z(\rho,t)=A^{\frac{1}{2}}u(t-\tau\rho),\quad\rho\in(0,1), t>0,$$
for system (\ref{s0}) and 
$$\ds z(\rho,t)=A^{\frac{\alpha}{2}}u(t-\tau\rho),\quad\rho\in(0,1), t>0,$$ for  system (\ref{s1}).

\medskip

Define
$$U=(u,u',\theta,z)^{\top},$$
then, problem (\ref{s0}) can be formulated as a first order system of the form

\begin{equation}\label{s7}
\left\{
\begin{array}{ll}
U^\prime (t)=\mathcal{A}_{\alpha,\beta}U(t), \, t > 0, & \\
\ds U(0)=\big(u_{0},u_{1},\theta_{0},\phi( - \tau\cdot)\big)^{\top},
\end{array}
\right.
\end{equation}
where the operator $\mathcal{A}_{\alpha,\beta}$ is defined by
$$\mathcal{A}_{\alpha,\beta}
\left(
\begin{array}{c}
u \\
v \\
\theta \\
z \\
\end{array}
\right)=
\left(
\begin{array}{c}
v \\
\ds-A^{\frac{1}{2}}\big(z(1)+aA^{\frac{1}{2}}v-A^{\beta-\frac{1}{2}}\theta\big) \\
-A^{\frac{\alpha}{2}}\left( A^{\frac{\alpha}{2}}\theta+A^{\beta-\frac{\alpha}{2}}v\right)  \\
\ds-\frac{1}{\tau}z_{\rho} \\
  \end{array}
\right),$$
with domain
$$D(\mathcal{A}_{\alpha,\beta})=\left\{
\begin{array}{c}
(u,v,\theta,z)^{\top}\in D(A^{\frac{1}{2}})\times D(A^{\frac{1}{2}})
\times D(A^{\frac{\alpha}{2}})\times H^{1}\big((0,1),H\big): \;  z(0)=A^{\frac{1}{2}}u, \\
\ds z(1)+aA^{\frac{1}{2}}v-A^{\beta-\frac{1}{2}}\theta \in D(A^\frac{1}{2}), \quad\textrm{ and }\quad A^{\frac{\alpha}{2}}\theta+A^{\beta-\frac{\alpha}{2}}v\in D(A^{\frac{\alpha}{2}})
\end{array}
\right\},$$
in the Hilbert space
$$\mathcal{H}=D(A^{\frac{1}{2}})\times H\times H\times L^{2}\big((0,1),H\big),$$
equipped with the scalar product
$$\ds\big((u,v,\theta,z)^{\top},(u_1,v_1,\theta_1,z_1)^{\top}\big)_{\mathcal{H}}
=\big(A^{\frac{1}{2}}u,A^{\frac{1}{2}}u_1\big)_{H}
+(v,v_1)_{H}+(\theta,\theta_1)_{H}+\xi\int_{0}^{1}(z,z_1)_{H}d\rho,$$
for some parameter $\xi>0$. 

\medskip

Similarly, define
$$U=(u,u',\theta,z)^{\top},$$
then, problem (\ref{s1}) can be formulated as a first order system of the form

\begin{equation}\label{s471}
\left\{
\begin{array}{ll}
U^\prime (t) =\mathcal{A}_{\alpha,\beta}U(t),\, t > 0, & \\
\ds U(0)=\big(x_{0},x_{1},\theta_{0},\psi(-\tau \cdot)\big)^{\top},
\end{array}
\right.
\end{equation}
where the operator $\mathcal{A}_{\alpha,\beta}$ is defined by
$$\mathcal{A}_{\alpha,\beta}
\left(
\begin{array}{c}
u \\
v \\
\theta \\
z \\
\end{array}
\right)= \left(
\begin{array}{c}
v \\
-A^{\frac{1}{2}}\left(A^{\frac{1}{2}} u-A^{\beta-\frac{1}{2}}\theta\right)  \\
\ds-\kappa A^{\frac{\alpha}{2}}\big(z(1)+\frac{a}
{\kappa}A^{\frac{\alpha}{2}}\theta+\frac{1}{\kappa} A^{\beta-\frac{\alpha}{2}}v\big) \\
\ds-\frac{1}{\tau}z_{\rho} \\
  \end{array}
\right),$$
with domain
$$D(\mathcal{A}_{\alpha,\beta})=\left\{
\begin{array}{c}
(u,v,\theta,z)^{\top}\in D(A^{\frac{1}{2}})\times D(A^{\frac{1}{2}})
\times D(A^{\frac{\alpha}{2}})\times H^{1}\big((0,1),H\big): \; z(0)=A^{\frac{\alpha}{2}}\theta, \\
\ds  A^{\frac{1}{2}} u-A^{\beta-\frac{1}{2}}\theta\in D(A^{\frac{1}{2}}), \quad\textrm{ and }\quad z(1)+\frac{a}
{\kappa}A^{\frac{\alpha}{2}}\theta+\frac{1}{\kappa} A^{\beta-\frac{\alpha}{2}}v\in D(A^{\frac{\alpha}{2}})
\end{array}
\right\},$$
in the Hilbert space
$$\mathcal{H}=D(A^{\frac{1}{2}})\times H\times H\times L^{2}\big((0,1),H\big),$$
equipped with the scalar product
$$\ds\big((u,v,\theta,z)^\top,(u_1,v_1,\theta_1,z_1)^\top\big)_{\mathcal{H}}
=\big(A^{\frac{1}{2}}u,A^{\frac{1}{2}}u_1\big)_{H}
+(v,v_1)_{H}+(\theta,\theta_1)_{H}+\xi\int_{0}^{1}(z,z_1)_{H}d\rho,$$
for some parameter $\xi>0$. 

\medskip

 In \cite{ASS24}, we proved that in the two cases the operator $\mathcal{A}_{\alpha,\beta}$ generates a $\mathcal{C}_0$-semigroup $e^{\mathcal{A}_{\alpha,\beta}t}$ and in particular, systems (\ref{s7}) and (\ref{s471}) are well-posed. Moreover, stability or regularity of the solution  $U(t)=e^{\mathcal{A}_{\alpha,\beta}t} U_0,\;t\geq 0$ correspond to stability and regularity of the corresponding semigroup.

 \medskip
 
 The stability of the $\mathcal{C}_0$-semigroup $e^{\mathcal{A}_{\alpha,\beta}t}$ was considered in the two cases in \cite{ASS24}. In this work we try to complete the study of qualitative properties of such semigroup by considering its regularity.  

\begin{figure}[th]
\centering
\includegraphics[width=7cm, keepaspectratio =true]{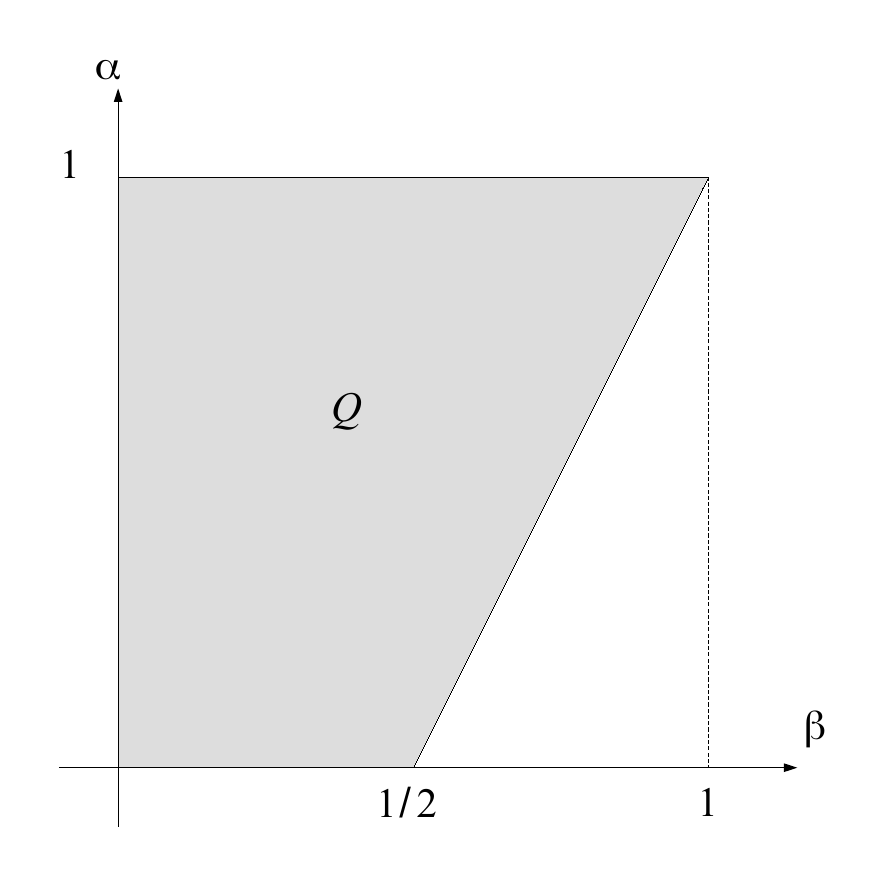}
\caption{Region $Q$}
\label{QQ}
\end{figure}

\medskip

Before starting, let us recall some definitions and characteristic properties about regularity of semigroups.

\begin{definition}
 Let $T(t):=e^{t\mathcal{A}}$ be a $C_0$-semigroup on a Hilbert space $\mathcal{H}_1$.
\begin{enumerate}
\item The semigroup $T(t)$ is said to be \textit{analytic} if
\begin{itemize}
\item for some $\varphi \in (0,\frac{\pi}{2})$, $T(t)$  can be extended to $\Sigma_\varphi$, where
$$\Sigma_\varphi=\{0\}\cup \{\tau \in \mathbb{C} \;:\;|\mathrm{arg}(\tau)|<\varphi \},$$
so that for any $x\in \mathcal{H}_1$, $t\mapsto T(t)x$ is continuous on $\Sigma_\varphi$, and for each $t_1,t_2 \in \Sigma_\varphi$, $T(t_1+t_2)=T(t_1)T(t_2)$. 
\item The map $t \mapsto T(t)$ is analytic over $\Sigma_\varphi\setminus \{0\}$, in the sense of the uniform operator topology of $\mathcal{L}(\mathcal{H}_1)$.
\end{itemize}

\item The semigroup $e^{t\mathcal{A}}$ is said to be of \textit{Gevrey} class $\delta$ (with $\delta >1$) if $e^{t\mathcal{A}}$ is infinitely differentiable and for any compact subset $\mathcal{K}\subset (0,\infty)$ and any $\lambda>0$
, there exists a constant $C=C(\lambda,\mathcal{K})$ such that
$$\|\mathcal{A}^ne^{t\mathcal{A}}\|_{\mathcal{L}(\mathcal{H}_1)}\leq C\lambda^n(n!)^\delta,\;\;\;\;\forall\;t\in\mathcal{K},\;n\geq 0.$$
\item The semigroup $e^{t\mathcal{A}}$ is said to be \textit{differentiable} if for any $x\in\mathcal{H}_1$, $t\mapsto e^{t\mathcal{A}}x$ is differentiable on $(0,\infty)$.
\end{enumerate}
\end{definition}

Well-posedness and stability of delayed PDEs (in infinite dimension) are considered along these years by many researchers \cite{Benhassi, ANbook, DLP86, Dat88, ADBB93, Dat97, BaPi05, NiPi06, JDM08, DQR09, Gug10, Rac12, ANP13}, see also \cite{ANP15, ASS24} and some references therein. However we found fewer studies about regularity properties of such type of PDEs, we can cite here \cite{BaPi01, Her02, HeVa03, Bat04, BaPi05}.

\medskip 

Returning to system (\ref{s0}) or (\ref{s1}). The case $(\tau=0,\; a=0)$: no delay and no damping term was first considered for regularity  in \cite{RiRa96} by Rivera and Racke. They proved that the associated semigroup of the $\alpha-\beta$ system, without delay and damping terms, is $\mathcal{C}^\infty$ in the region
$$\mathring{S}:=\{(\beta,\alpha)\in[0,1]\times [0,1]\mid |2\beta-1|<\alpha< 2\beta\}.$$ Later on, Liu and Yong \cite{LiYo98} proved the analyticity of such semigroup in region $R_1$, and of Gevrey class $\delta>\frac{1}{2(2\beta-\alpha)}$ in the region $R_2$ (the region $R_1$ and $R_2$ will be defined below). Then,  in  \cite{HLY15}, Hao, Liu and Yong complete the regularity analysis. Furthermore they gave a summary about stability and regularity of the $\alpha-\beta$ system for $(\beta,\alpha)$ in $[0,1]\times [0,1]$. For this, we need the following two partitions of the unit square $[0,1]\times [0,1]$ (Figure \ref{A1} and Figure \ref{A2}):
\begin{equation*}
\left\{
\begin{array}{lll}
S &= &\{(\beta,\alpha)\in[0,1]\times [0,1]\mid |2\beta-1|\leq\alpha\leq 2\beta\}, \\
S_1&= &\{(\beta,\alpha)\in[0,1]\times [0,1]\mid 2\beta<\alpha,\;\frac{1}{2}<\alpha\}, \\
S_2&= &\{(\beta,\alpha)\in[0,1]\times [0,1]\mid \alpha<1-2\beta,\;\alpha\leq\frac{1}{2}\},\\
S_3&= &\{(\beta,\alpha)\in[0,1]\times [0,1]\mid \alpha<2\beta-1\}
\end{array}
\right.
\end{equation*}
and 
\begin{equation*}
\left\{
\begin{array}{lll}
R_1 &= &\{(\beta,\alpha)\in[0,1]\times [0,1]\mid \beta\leq\alpha\leq 2\beta-\frac{1}{2}\}, \\
R_2&= &\{(\beta,\alpha)\in[0,1]\times [0,1]\mid (2\beta-\frac{1}{2})<\alpha,\;\frac{1}{2}<\alpha\;\text{and}\;\alpha<2\beta\}, \\
R_3&= &\{(\beta,\alpha)\in[0,1]\times [0,1]\mid 0\leq1-2\beta<\alpha\leq\frac{1}{2},\;(\beta,\alpha)\neq(\frac{1}{2},\frac{1}{2})\}, \\
R_4&= &\{(\beta,\alpha)\in[0,1]\times [0,1]\mid 0<2\beta-1\leq\alpha<\beta\},\\
R_5&= & \{(\beta,\alpha)\in[0,1]\times [0,1]\mid 0<\alpha<2\beta-1\}= S_3\setminus S_I,\\
S_I&= &\{0\}\times(\frac{1}{2},1] .
\end{array}
\right.
\end{equation*}
\begin{figure}[tbp]
\centering
\begin{minipage}[t]{7cm}
\centering
\includegraphics[width=7cm, keepaspectratio =true]{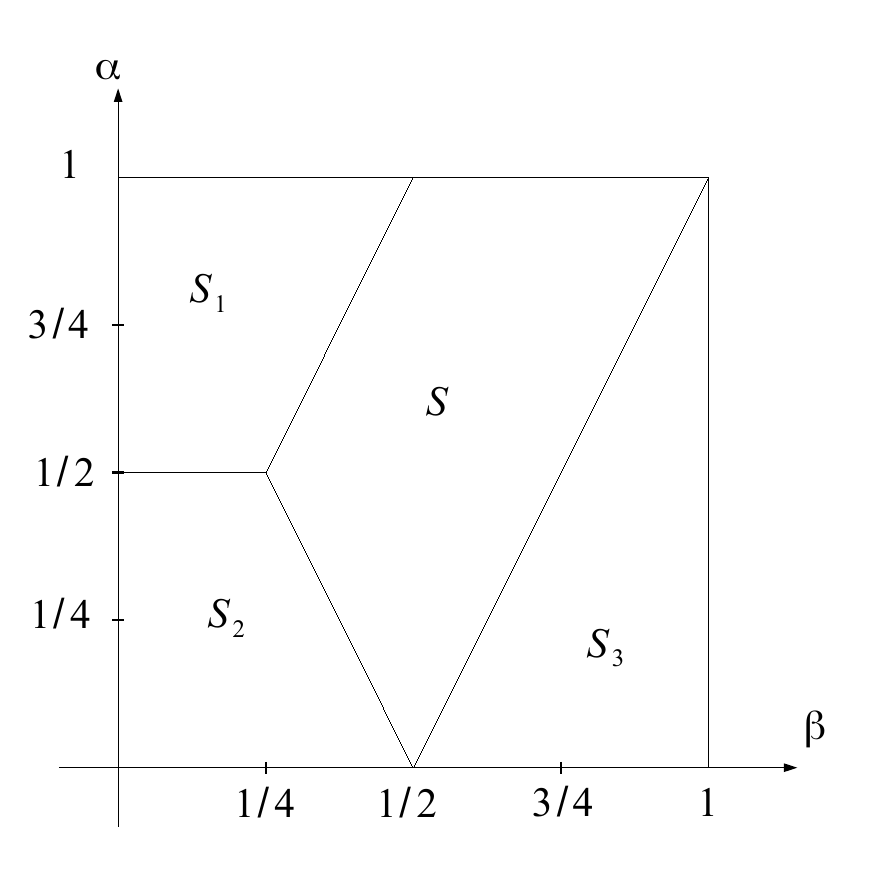}
\caption{Region of stability (without delay)}
\label{A1}
\hfill
\end{minipage}
\begin{minipage}[t]{7cm}
\centering
\includegraphics[width=7cm, keepaspectratio =true]{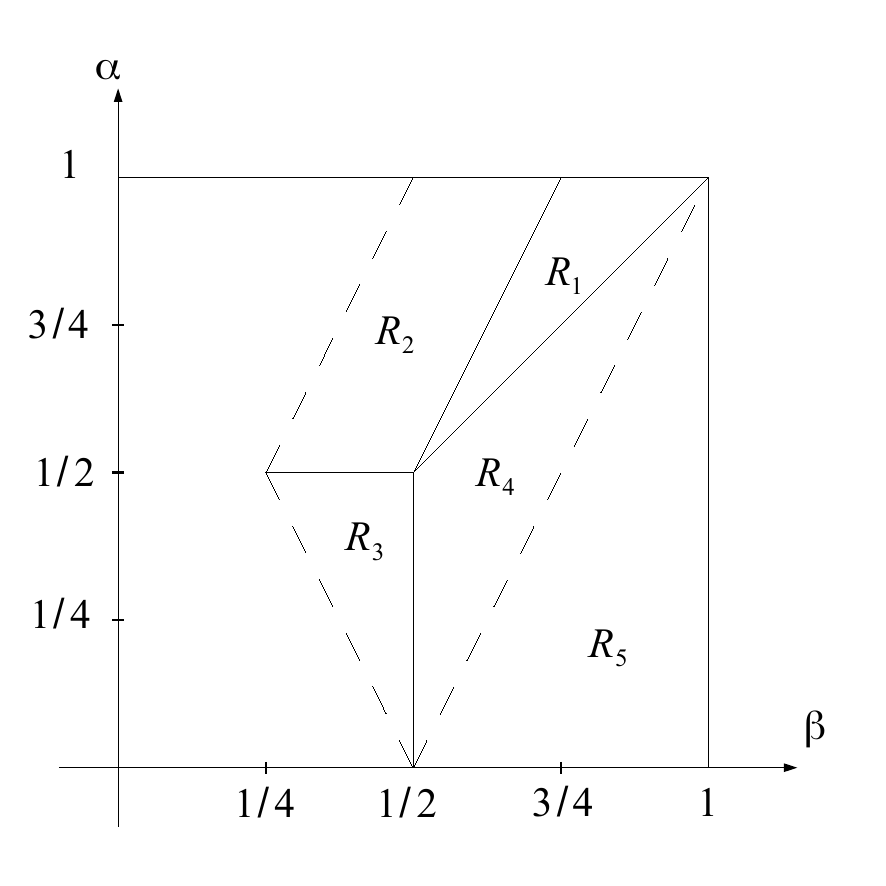}
\caption{Region of regularity (
without delay)}
\label{A2}
\end{minipage}
\end{figure}
\medskip

The table below about qualitative properties of an $\alpha-\beta$-system can be found in \cite{HLY15}:

\begin{tabular}{lll}
\\
\hline

Regions & Regularity & Stability\\

\hline
\\
$R_1$ & analytic & exponentially stable\\
$R_2$ &Gevrey class $\delta>\frac{1}{2(2\beta-\alpha)}$ & exponentially stable\\
$R_3$ &Gevrey class $\delta>\frac{1}{2(2\beta+\alpha)-2}$ & exponentially stable\\
$R_4$ &Gevrey class $\delta>\frac{\beta}{\alpha}$ & exponentially stable\\
$R_5$ &Gevrey class $\delta>\frac{\beta}{\alpha}$ & not asymtotically stable\\
$S_I$ &not differentiable & not asymtotically stable\\
$\overline{(S_1\cup S_2)}\cap S  $ &not differentiable & exponentially stable\\
$S_1$ &not differentiable & polynomially stable of order $\frac{1}{2(\alpha-2\beta)}$\\
$S_2$ &not differentiable & polynomially stable of order $\frac{1}{2-2(2\beta+\alpha)}$\\
\hline
\\
\end{tabular}

The stability analysis at region $S_3$ was completed recently by Ammari et \textit{al.} \cite{ALS23}, by proving a  new result due to \cite{BCT16}. Finally, note that regularity of more general or related systems (without delay) was taken into account since the 80s \cite{ChTr89, ChTr90, LiYo98, LaTr98}. See also the recent works \cite{KLS21} and references therein. 

\medskip

The case $(\tau>0, a>0)$ was considered in \cite{ASS24} for stability, we have proved that the presence of the damping terms $aAu^\prime(t)$ and $aA^\alpha\theta(t)$  allowed us to recover well-posedness and stability results obtained for no delayed corresponding systems. Moreover, the damping $aAu^\prime(t)$ is so strong that it makes the system (\ref{s0}) exponentially stable even for $(\beta,\alpha)$ in region $S_1\cup S_2$:
\begin{theorem} \cite{ASS24}\label{th1.2}

\begin{enumerate}
\item Assume that $a\geq\tau$. Then for  $\xi \geq \frac{2\tau}{a}$,  the $\mathcal{C}_0$-semigroup $e^{\mathcal{A}_{\alpha,\beta}t}$ associated to system (\ref{s0}) is exponentially stable in the region
$$Q:=S\cup S_1\cup 
S_2 = \{(\beta,\alpha)\in[0,1]\times [0,1]\mid 2\beta-\alpha\leq 1\}.$$
\item Assume that $a>\kappa$. Then for $\xi\in\mathring{\mathcal{J}}_{a,\kappa,\tau}:=\big]\tau\big(a-\sqrt{a^2-\kappa^2}\big),\tau\big(a+\sqrt{a^2-\kappa^2}\big)\big[$,   the semigroup $e^{t\mathcal{A}_{\alpha,\beta}}$ associated to system (\ref{s1}) has the following stability properties:

(i) In $S$, it is exponentially stable; 

(ii) In $S_1 $, it is polynomially stable of order $\frac{1}{2(\alpha-2\beta)}$;

(iii) In $S_2$, it is polynomially stable of order $\frac{1}{2-2(\alpha+2\beta)}$.

\end{enumerate}
\end{theorem} 

Our main question in the present work, is whether damping terms also guarantee others qualitative properties (such as analyticity or differentiability) as in the case of non delayed corresponding systems. The answer is negative, Moreover, we will prove that $e^{t\mathcal{A}_{\alpha,\beta}}$ is not even immediately norm-continuous in all the region $Q$ in the case of system (\ref{s0}).
\begin{definition}
 A $C_0$-semigroup $T(t):=e^{t\mathcal{A}}$ on a Hilbert space $\mathcal{H}_1$ is said to be \textit{immediately norm-continuous} if the function $$t\mapsto T(t)$$ is norm-continuous  from $(0,\infty)$ into $\mathcal{L}(\mathcal{H}_1)$.
\end{definition}
Note that if a semigroup is differentiable then it is immediately norm-continuous.

\medskip

In this work we will use the following characterization of an immediately norm-continuous semigroup in term of the resolvent \cite{Puh92, ElEn96, EnNa00}.
 
\begin{lemma}\label{cnc}
Let $\mathcal{A}$ be the generator of a strongly continuous, exponentially stable semigroup $(T(t))_{t\geq 0}$ on a Hilbert space $\mathcal{H}_1$ . Then $(T(t))_{t\geq 0}$ is immediately  norm-continuous if and only if
\begin{equation}\label{nc}
\lim_{\lambda \in \mathbb{R},\,|\lambda|\rightarrow\infty}||(i\lambda I - A)^{-1}||_{\mathcal{L}(\mathcal{H}_1)}=0.
\end{equation}
\end{lemma}
Such characterization is similar to those done for exponential and polynomial stability \cite{Gea78, Pruss, Hua85, BoTo10}), or analyticity \cite[Chap. 1, p. 5]{LiZh99}, of Gevrey class $\delta>0$ \cite{Tay89} and differentiability  \cite[Chap. 2, p. 57]{Paz83}, of $\mathcal{C}_0$-semigroups.

\begin{remarks}\label{rms}
\begin{enumerate}
\item It can be proved that for a $C_0$-semigroup $T(t):=e^{t\mathcal{A}}$ on a Hilbert space $\mathcal{H}_1$ one has the following implications:
$$\text{analytic}\;\;\Longrightarrow\;\; \text{differentiable}\;\;\Longrightarrow\;\;\text{immediately norm-continuous}.$$
\item If  $e^{t\mathcal{A}}$ is bounded and $\mathbf{i}\mathbb{R}\subset \rho(\mathcal{A})$ then we have the following implication for $e^{t\mathcal{A}}$:
$$\text{differentiable}\;\;\Longrightarrow\;\;\text{exponentially stable}.$$
\end{enumerate}

\end{remarks}

The paper is organized as follows. In the second section we prove that the abstract thermoelastic system is not immediately norm-continuous when the delay is present at the hyperboloc part (\ref{s0}), and in the third section we prove that the system is non differentiable if the delay is present at the parabolic part (\ref{s1}). we end each section by some related systems or applications.

\medskip

In the sequel,  $\fbox{$(\beta,\alpha)\in Q:=\{(\beta,\alpha)\in[0,1]\times [0,1]\mid 2\beta-\alpha\leq 1\}.$}$

\section{Delay at the parabolic part (system (\ref{s0}))}

The main result of this section is to prove that the semigroup $e^{\mathcal{A}_{\alpha,\beta}t}$ corresponding to system (\ref{s0}) is not immediately norm-continuous (and so non differentiable).

\subsection{{Lack of immediate norm-continuity of semigroup associated to (\ref{s0})}}

 \begin{theorem}\rm\label{theorem21}
For $a\geq\tau$ and even if
$\ds\xi\geq\frac{2\tau}{a}$, the semigroup $e^{\mathcal{A}_{\alpha,\beta}t}$ is not immediately norm-continuous in all the region $Q$. In particular $e^{\mathcal{A}_{\alpha,\beta}t}$ is not differentiable for $(\beta,\alpha)\in Q$.
\end{theorem}

\begin{proof}
Using Theorem \ref{th1.2} and Lemma \ref{cnc}, we will show that there exists a sequence of positive real numbers $\lambda_n$ and for each $n$, an element $F_n \in \mathcal{H}$  such that $\lim_{n\rightarrow\infty}\lambda_n=\infty,\; ||F_n||=1$ and the solution $U_n$ of
\begin{equation}\label{e1}
(i\lambda_n I - \mathcal{A}_{\beta,\alpha})U_n=F_n
\end{equation}
satisfies
$$\lim_{n\rightarrow\infty}||U_n||=c>0.$$
 

 which would establish the claimed result.

For each $n\in\mathbb{N}$, we choose $F_n$ of the form $F_n=(0,0,0,h_n)^\top$ where $h_n \in L^{2}\big((0,1),H\big) $ will be chosen later.

The solution $U_n=(u_n,v_n,\theta_n,z_n)^\top$ of (\ref{e1}) should satisfies 

\begin{equation}\label{s28}
\mathbf{i}\lambda_nA^{\frac{1}{2}}u_n-A^{\frac{1}{2}}v_n =0 \textrm{ in } H,
\end{equation}

\begin{equation}\label{s29}
\mathbf{i}\lambda_nv_n+A^{\frac{1}{2}}\big(z_n(1)+aA^{\frac{1}{2}}v_n
-A^{\beta-\frac{1}{2}}\theta_n\big)=0 \textrm{ in } H,
\end{equation}

\begin{equation}\label{s30}
\mathbf{i}\lambda_n\theta_n+A^{\frac{\alpha}{2}}\big(A^{\frac{\alpha}{2}}\theta_n+A^{\beta-\frac{\alpha}{2}}v_n\big)=0 \textrm{ in } H,
\end{equation}

\begin{equation}\label{s31}
\mathbf{i}\lambda_nz_n+\frac{1}{\tau}\partial_{\rho}z_n=h_n \textrm{ in }L^2\big((0,1)H\big).
\end{equation}
By integration of the identity (\ref{s31}), we obtain
\begin{equation}\label{s31'}
\ds z_n(\rho)=A^{\frac{1}{2}}u_n e^{-\mathbf{i}\lambda_n\tau\rho}+
\tau e^{-\mathbf{i}\tau\lambda_n\rho}\int_{0}^{\rho}e^{\mathbf{i}\tau\lambda_ns}h_n(s)ds.
\end{equation}
From (\ref{s29}) we deduce that
\begin{equation}\label{s29'}
z_n(1)=-\left( \mathbf{i}\lambda_nA^{-\frac{1}{2}}v_n+aA^{\frac{1}{2}}v_n
-A^{\beta-\frac{1}{2}}\theta_n\right).
\end{equation}
Taking $\rho=1$ in (\ref{s31'}) and comparing the obtained result to (\ref{s29'}) to get
\begin{equation}\label{s31''}
\tau e^{-\mathbf{i}\tau\lambda_n}\int_{0}^{1}e^{\mathbf{i}\tau\lambda_ns}h_n(s)ds=\Phi_n.
\end{equation}
where
\begin{equation}\label{phi}
\Phi_n=-\left(A^{\frac{1}{2}}u_n e^{-\mathbf{i}\tau\lambda_n}+ \mathbf{i}\lambda_nA^{-\frac{1}{2}}v_n+aA^{\frac{1}{2}}v_n
-A^{\beta-\frac{1}{2}}\theta_n\right).
\end{equation}
We choose $$h_n(s):=\frac{1}{\tau}e^{\mathbf{i}\tau\lambda_n(1-s)}\Phi_n$$
which guaranties the equality (\ref{s31''}) and gives
 \begin{equation*}
\ds z_n(\rho)=A^{\frac{1}{2}}u_n e^{-\mathbf{i}\lambda_n\tau\rho}+
\rho e^{\mathbf{i}\tau\lambda_n(1-\rho)}\Phi_n.
\end{equation*}

For each $n\in \mathbb{N}$, let $e_n$ an eigenvector of $A$ with $\|e_n\|=1$ and $Ae_n=\mu_ne_n$, where $\mu_n$ is a sequence of real numbers such that $\lim\limits_{n\rightarrow\infty}\mu_n=\infty$. Let us take
\begin{equation}\label{s31bisbis}
\theta_n=\frac{1}{\mu_n^p}e_n,\;\;\lambda_n=\mu_n^q
\end{equation}
where $p$ and $q$ are real positive numbers to be specified later,  such that $U_n$ and $F_n$ fulfill the desired conditions.

It follows from (\ref{s28})-(\ref{s30}), (\ref{s31bisbis}) and (\ref{phi}) that:
\begin{eqnarray}
v_n&=&\left(i\mu_n^{-\beta-p+q}+\mu_n^{\alpha-\beta-p}\right)e_n, \label{2.11}\\
A^{\frac{1}{2}}u_n&=&\left(\mu_n^{\frac{1}{2}-\beta-p}-i\mu_n^{\frac{1}{2}+\alpha-\beta-p-q} \right)e_n,
\end{eqnarray}
and
\begin{eqnarray}
-\Phi_n&=&\left[ \left( \mu_n^{\frac{1}{2}-\beta-p}-i\mu_n^{\frac{1}{2}+\alpha-\beta-p-q}\right) e^{-i\lambda_n\tau}-\mu_n^{-\frac{1}{2}-\beta-p+2q}+i\mu_n^{-\frac{1}{2}+\alpha-\beta-p+q}\right. \notag\\
&+&\left. ia\mu_n^{\frac{1}{2}-\beta-p+q}+a\mu_n^{\frac{1}{2}+\alpha-\beta-p}-\mu_n^{\beta-\frac{1}{2}-p}\right] e_n.\label{2.14}
\end{eqnarray}

Case 1: $\alpha>0$. Note that $\frac{1}{2}-\beta+\alpha\geq \frac{\alpha}{2}\geq 0$. Then,  choosing $p=\frac{1}{2}-\beta+\alpha$ and $q=\alpha$,  we obtain, using (\ref{2.11})-(\ref{2.14}),
\begin{eqnarray*}
\theta_n&=& o(1),\\ 
v_n&=&\left(i+1\right)\mu_n^{-\frac{1}{2}}e_n=o(1),\\
A^{\frac{1}{2}}u_n&=&\left(1-i\right)\mu_n^{-\alpha}e_n=o(1),
\end{eqnarray*}
\begin{eqnarray*}
-\Phi_n&=&\left(\left(i-1\right)\mu_n^{\alpha-1}+a\left(i+1\right)+o(1)\right) e_n.
\end{eqnarray*}
One checks that $h_n\in L^2\big((0,1)H\big) $, $\|U_n\|\rightarrow c_1>0$, and $\|F_n\|\sim c_2>0$ for some positive constants $c_1$ and $c_2$.\\
Finally, recall that $U_n$ is solution of (\ref{e1}). Thus, the proof is complete by choosing $\frac{F_n}{\|F_n\|}$ instead of $F_n$.

\medskip

Case 2:
 $\alpha=0$. Since $(\beta,\alpha)\in Q$, we have $\beta\leq \frac{1}{2}$. Then we choose $p=\frac{1}{2}-\beta+\delta,\;q=\delta$, where $\delta$ is a small positive number. We obtain the same conclusion.
\end{proof}

\subsection{Application}{ \textit{Thermoelastic plate with delay}}\label{sec1}

Taking $\alpha=\beta=\frac{1}{2}$, $H=L^2(\Omega)$ where $\Omega$ is a smooth open bounded domain in $\mathbb{R}^n$,  and consider  

\begin{equation}\label{exp11}
\left\{
\begin{array}{ll}
u_{tt}(x,t)+\Delta^2u(x,t-\tau)+a \Delta^2u_{t}(x,t)+\Delta \theta(x,t)=0, &
\quad (x,t)\in\Omega\times(0,+\infty), \\
\theta_{t}(x,t)-\Delta\theta(x,t)-\Delta u_{t}(x,t)=0, &
\quad(x,t)\in\Omega\times(0,+\infty), \\
u(x,t)=\Delta u(x,t)=0, & \quad (x,t)\in\partial \Omega\times(0,+\infty), \\
u(x,0)=u^0(x),u_t(x,0)=u^1(x), & \quad t\in(0,\tau), \\
\theta(x,t)=0, & \quad (x,t)\in\partial \Omega\times(0,+\infty), \\
\theta(x,0)=\theta^0(x), &  \\
-\Delta u(x,t)=f_0(x,t), &\quad-\tau\leq t<0, x\in\Omega,
\end{array}
\right.
\end{equation}
where $\tau$ and $a$ are real positive constants.

\medskip

Here, $A^{\frac{1}{2}}=-\Delta$, with domain $D(A^{\frac{1}{2}})=H^2(\Omega)\cap H^1_0(\Omega)$, and $A=-\Delta^2$, with domain $D(A)=H^4(\Omega)\cap H^2_0(\Omega).$

\medskip

The operator $\mathcal{A}_{\frac{1}{2},\frac{1}{2}}$ is given by
\begin{equation*}
\ds\mathcal{A}_{\frac{1}{2},\frac{1}{2}}
\left(
\begin{array}{c}
u \\
v \\
\theta \\
z \\
\end{array}
\right)=
\left(
\begin{array}{c}
v \\
\ds-\Delta\big(-z(1)+a\Delta v\big)-\Delta \theta \\
\Delta\theta+\Delta v \\
\ds-\frac{1}{\tau}z_{\rho} \\
  \end{array}
\right),
\end{equation*}
with domain
\begin{equation*}
D(\mathcal{A}_{\frac{1}{2},\frac{1}{2}})=\left\{
\begin{array}{c}
(u,v,\theta,z)^{\top}\in  \left( H_{0}^{1}\big(\Omega\big)\cap H^{2}\big(\Omega\big)\right)^3 \times H^{1}\big((0,1),L^{2}(\Omega)\big): \\
\ds \quad z(0)=-\Delta u\quad\textrm{ and }\quad z(1)-a\Delta v\in H_{0}^{1}\big(\Omega\big)\cap H^{2}\big(\Omega\big)
\end{array}
\right\},
\end{equation*}
in the Hilbert space
\begin{equation*}
\mathcal{H}=\left( H_{0}^{1}\big(\Omega\big)\cap H^{2}\big(\Omega\big)\right)\times
L^{2}\big(\Omega\big)\times L^{2}\big(\Omega\big)\times
L^{2}\big(\Omega\times H^{1}(0,1)\big).
\end{equation*}



 Applying Theorem \ref{theorem21} with
$\ds (\beta,\alpha)=\left(\frac{1}{2},\frac{1}{2}\right)\in Q$,
one has the following result.

\begin{theorem}
Even if $a\geq \tau$ and $\ds\xi\geq\frac{2\tau}{a}$, the semigroup $e^{\mathcal{A}_{\frac{1}{2},\frac{1}{2}}t}$ associated to (\ref{exp11}) is not immediately norm-continuous and then it is not differentiable.
\end{theorem}

\subsection{Some related systems}
\subsubsection{Exemple 1}
We consider the following system
 \begin{equation}\label{s0'"}
\left\{
\begin{array}{ll}
u''(t)+ Au(t-\tau)+aAu'(t)-A^\beta\theta(t)=0, & \quad t\in(0,+\infty), \\
\theta'(t)+A^{\alpha}\theta(t)+A^\beta u'(t)=0, & \quad t\in(0,+\infty), \\
u(0)=u_{0}, u'(0)=u_{1}, \theta(0)=\theta_{0}, &  \\
\ds B^{*}u(t-\tau)=\phi(t-\tau), & \quad t\in(0,\tau),
\end{array}
\right.
\end{equation}
where $\tau>0$ is a constant time delay, $a>0$ and $B:D(B) \subset H\rightarrow H$ is a closed densely defined linear operator with $B^*$ is the adjoint of $B$. 
 We suppose that $A=BB^*$ and $B^*$ has the following properties:
 \begin{equation*}
  D(B^{*})= D(A^{\frac{1}{2}}) 
 \end{equation*}
 \begin{equation*}
\exists \, c_1 > 0 \; \hbox{ s.t. } \; \|A^{\frac{1}{2}}v\|_H\leq c_1\|B^{*}v\|_{H},\;\;\forall\
v\in D(B^{*}),
\end{equation*}  
 and 
 \begin{equation*}
 \beta\leq \alpha, \quad \beta\leq \frac{1}{2}.
 \end{equation*}
 Note that such system corresponds to that of reference (2.28) in \cite{ASS24}, with $C=A^\beta$ and satisfies all its hypothesis, in fact one can easily  check that: $$D(A^\alpha)\subset D(A^\beta),\;\;\;D(B^*)\subset  D(A^\beta)\;\;\;$$
 and
 \begin{equation*}
 \exists \, c_2 > 0 \; \hbox{ s.t. } \; \|A^\beta v\|_H\leq c_2\|B^*v\|_{H},\;\;\forall\
v\in D(B^{*}).
\end{equation*}  

So, define
$$\ds z(\rho,t):=B^{*}u(t-\tau\rho),\quad\rho\in(0,1), t>0.$$
then, by taking $U=(u,u',\theta,z)^{\top}$,  problem (\ref{s0'"}) is equivalent to

\begin{equation}\label{s77}
\left\{
\begin{array}{ll}
U'=\mathcal{A}U, & \\
\ds U(0)=\big(u_{0},u_{1},\theta_{0},\phi(-\tau.)\big)^{\top},
\end{array}
\right.
\end{equation}
where the operator $\mathcal{A}$ is defined by
$$\mathcal{A}
\left(
\begin{array}{c}
u \\
v \\
\theta \\
z \\
\end{array}
\right)=
\left(
\begin{array}{c}
v \\
\ds-B\big(z(1)+aB^*v\big)+A^\beta\theta \\
-A^{\alpha}\theta-A^{\beta} v \\
\ds-\frac{1}{\tau}z_{\rho} \\
  \end{array}
\right),$$
with domain
$$D(\mathcal{A})=\left\{
\begin{array}{c}
(u,v,\theta,z)^{\top}\in D(B^*)\times D(B^*)
\times D(A^{\alpha})\times H^{1}\big((0,1),H\big): \\
\ds z(0)=B^*u\quad\textrm{ and }\quad aB^*v+z(1)\in D(B)
\end{array}
\right\},$$
in the Hilbert space
$$\mathcal{H}=D(B^*)\times H\times H\times L^{2}\big((0,1),H\big),$$
equipped with the scalar product
$$\ds\big((u,v,\theta,z)^\top,(u_1,v_1,\theta_1,z_1)^\top\big)_{\mathcal{H}}
=\big(B^*u,B^*u_1\big)_{H}
+(v,v_1)_{H}+(\theta,\theta_1)_{H}+\xi\int_{0}^{1}(z,z_1)_{H}d\rho,$$
where $\xi>0$ is a positive constant to be specified.

Recall that, we have proved in \cite{ASS24} that, for  $a\geq \tau$ and $\ds\xi\geq\frac{2\tau}{a}$ the operator $\mathcal{A}$ generates a $\mathcal{C}_0$-semigroup on $\mathcal{H}$, which is exponentially stable. So, 
using Lemma \ref{cnc}, we deduce as in the initial system:
\medskip
\begin{theorem}\label{rel1}
Even if $a\geq \tau$ and $\ds\xi\geq\frac{2\tau}{a}$, the semigroup associated to system (\ref{s77}) is  not immediately norm-continuous, then it is not differentiable.
\end{theorem}

\subsubsection{Exemple 2: Thermoelastic string system with delay}
 
\begin{equation}\label{s84""}
\left\{
\begin{array}{ll}
u_{tt}(t,x)-u_{xx}(x,t)-au_{xxt}(x,t-\tau)+\theta_x(x,t)=0, &
\quad (x,t)\in(0,\pi)\times(0,+\infty), \\
\theta_{t}(x,t)-\theta_{xx}(x,t)+u_{xt}(x,t)=0, &
\quad(x,t)\in(0,\pi)\times(0,+\infty), \\
u(0,t)=u(\pi,t)=0, & \quad t\in(0,+\infty), \\
u(x,0)=u^0(x),u_t(x,0)=u^1(x), & \quad t\in(0,\tau), \\
\theta_x(0,t)=\theta_x(\pi,t)=0, & \quad t\in(0,+\infty), \\
\theta(x,0)=\theta^0(x), &  \\
u_x(x,t)=f_0(x,t), &\quad-\tau\leq t<0, x\in (0,\pi).
\end{array}
\right.
\end{equation}

Here, we assume that $\int_\Omega \theta(x,t)dx=0$. Otherwise, we can make the substitution $\tilde{\theta}(x,t)=\theta (x,t)-\frac{1}{\ell}\int_\Omega \theta_0(x)dx,$ in fact $(u,v,z,\theta)$ and $(u,v,z,\tilde{\theta})$ satisfy the same system (\ref{s84""}).

The Hilbert space state will be 
\begin{equation*}
\mathcal{H}=\left\{(f,g,h,p)\in H_{0}^{1}\big(0,\pi\big)\times
L^{2}\big(0,\pi\big)\times L^{2}\big(0,\pi\big)\times
L^{2}\big((0,1), L^2(0,\pi)\big)\mid \int_0^\pi h(x)dx=0\right\}.
\end{equation*}
Define
$$\ds z(\rho,t):=u_x(t-\tau\rho),\quad\rho\in(0,1), t>0.$$
then, by taking $U=(u,u',\theta,z)^{\top}$,  problem (\ref{s84""}) is equivalent to

\begin{equation*}
\left\{
\begin{array}{ll}
U'(t)=\mathcal{A}U(t), \, t > 0, & \\
\ds U(0)=\big(u_{0},u_{1},\theta_{0},f_0(\cdot,-\tau\cdot)\big)^{\top},
\end{array}
\right.
\end{equation*}
where  the operator $\mathcal{A}$ is simply written as follows
\begin{equation*}
\ds\mathcal{A}
\left(
\begin{array}{c}
u \\
v \\
\theta \\
z \\
\end{array}
\right)=
\left(
\begin{array}{c}
v \\
\ds \left( z(1)+av_x\right)_x-\theta_x \\
\theta_{xx}-v_x \\
\ds-\frac{1}{\tau}z_{\rho} \\
  \end{array}
\right),
\end{equation*}
with domain\; $D(\mathcal{A})=$
\begin{equation*}
\left\{
\begin{array}{c}
(u,v,\theta,z)^{\top}\in  H_{0}^{1}\big(0,\pi\big) \times H_{0}^{1}\big(0,\pi\big)
\times  \left(H_{0}^{1}\big(0,\pi\big)\cap H^{2}\big(0,\pi\big)\right) \times H^{1}\big((0,1),L^2(0,\pi)\big): \\
\ds \quad z(0)=u_x,\quad\theta_x(0)=\theta_x(\pi)=0,\quad\textrm{ and }\quad z(1)+av_x\in H^{1}\big(0,\pi\big)
\end{array}
\right\},
\end{equation*}
in the Hilbert space $\mathcal{H}$.

\medskip

 Recall that $\mathcal{A}$ generates, on $\mathcal{H}$, a $\mathcal{C}_0$-semigroup  for $\xi\geq\frac{2\tau}{a}$, which is exponentially stable for $a\geq \tau$ \cite{ASS24, KhSh24}. But $e^{\mathcal{A}t}$ is not differentiable,  or more precisely, we have the following result.

\begin{theorem}
Even if $a\geq \tau$, the semigroup $e^{\mathcal{A}t}$ associated to (\ref{s84""}) is not immediately norm-continuous.
\end{theorem} 
\begin{proof}
We use again Theorem \ref{th1.2} and Lemma \ref{cnc}. So we will built, for every $n\in\mathbb{N}$, $\lambda_n>0$ and $F_n \in \mathcal{H}$  such that $\ds \lim_{n\rightarrow\infty}\lambda_n=\infty,\; ||F_n||=1$ and the solution $U_n$ of
\begin{equation}\label{e1r}
(i\lambda_n I - \mathcal{A}_{\beta,\alpha})U_n=F_n
\end{equation}
satisfies
$$\lim_{n\rightarrow\infty}||U_n||=c>0.$$
 

 which would establish the claimed result.

\medskip

For each odd number $n\in\mathbb{N}$, we choose $F_n$ of the form $F_n=(0,0,0,h_n)^\top$ where $h_n \in L^{2}\big((0,1),L^2(0,\pi)\big) $ will be chosen later.

\medskip

The solution $U_n=(u_n,v_n,\theta_n,z_n)^\top$ of (\ref{e1r}) should satisfies 

\begin{equation}\label{s28r}
\mathbf{i}\lambda_nu_{n,x}-v_{n,x} =0 \textrm{ in } L^2(0,\pi),
\end{equation}

\begin{equation}\label{s29r}
\mathbf{i}\lambda_nv_n-\big(z_n(1)+av_{n,x}\big)_x
+\theta_{n,x}=0 \textrm{ in } L^2(0,\pi),
\end{equation}

\begin{equation}\label{s30r}
\mathbf{i}\lambda_n\theta_n-\theta_{n,xx}+v_{n,x}=0 \textrm{ in } L^2(0,\pi),
\end{equation}

\begin{equation}\label{s31r}
\mathbf{i}\lambda_nz_n+\frac{1}{\tau}\partial_{\rho}z_n=h_n \textrm{ in }L^2\big((0,1),L^2(0,\pi)\big).
\end{equation}

By choosing
\begin{equation}\label{sr1}
\theta_n(x)=\frac{1}{n^2}\cos (nx)\;\;\;\;\text{and}\;\;\;\;\lambda_n=n^2,
\end{equation} 
one has $u_n,v_n\in H_0^1(0,\pi)\cap H^2(0,\pi)$ and 
\begin{equation}\label{sr2}
v_n(x)=\frac{1+\mathbf{i}}{n}\sin(nx),\;\;\;\;\;\;u_{n,xx}(x)=\frac{1-\mathbf{i}}{n}\sin(nx).
\end{equation}

\medskip

Now, from (\ref{s30r}) we have
\begin{equation}\label{sr}
z_n(1)_x=\mathbf{i}\lambda_n v_{n}-av_{n,xx}+\theta_{n,x}.
\end{equation}
 
On the other hand,
by integration of the identity (\ref{s31r}), we get
\begin{equation}\label{s31'r}
\ds z_n(\rho)=u_{n,x} e^{-\mathbf{i}\lambda_n\tau\rho}+
\tau e^{-\mathbf{i}\tau\lambda_n\rho}\int_{0}^{\rho}e^{\mathbf{i}\tau\lambda_ns}h_n(s)ds.
\end{equation}
In particular,
\begin{equation}\label{s31'rr}
\ds z_n(1)_x=u_{n,xx} e^{-\mathbf{i}\lambda_n\tau}+
\tau e^{-\mathbf{i}\tau\lambda_n}\left( \int_{0}^{1}e^{\mathbf{i}\tau\lambda_ns}h_n(s)ds\right)_x .
\end{equation}
Hence, by comparing  (\ref{sr}) and (\ref{s31'rr}), we obtain
\begin{equation}
\tau e^{-\mathbf{i}\tau\lambda_n} \int_{0}^{1}e^{\mathbf{i}\tau\lambda_ns}h_n(s)ds=\Phi_n(x)
\end{equation}
where
\begin{equation}\label{phir}
\Phi_n(x):=\mathbf{i}\lambda_n \int_{\frac{\pi}{2}}^xv_{n}(s)ds-av_{n,x}+\theta_{n}-u_{n,x} e^{-\mathbf{i}\lambda_n\tau}
\end{equation} 

We choose $$h_n(s):=\frac{1}{\tau}e^{\mathbf{i}\tau\lambda_n(1-s)}\Phi_n$$
which gives
 \begin{equation*}\label{s31bisr}
\ds z_n(\rho)=u_{n,x} e^{-\mathbf{i}\lambda_n\tau\rho}+
\rho e^{\mathbf{i}\tau\lambda_n(1-\rho)}\Phi_n.
\end{equation*}

It follows from (\ref{sr1}), (\ref{sr2})  and (\ref{phir}) that:

\begin{eqnarray*}
\Phi_n=\left( -\mathbf{i}(1+\mathbf{i})-a(1+\mathbf{i})+\frac{1}{n^2}-\frac{\mathbf{i}(1+\mathbf{i})}{n^2}\right)  \cos(nx)\\
=\left(  -\mathbf{i}(a+1)+1-a+o(1)\right) \cos(nx).
\end{eqnarray*}

One checks that $h_n\in L^2\big((0,1),L^2(0,\pi)\big) $, $U_n$ is actually solution of (\ref{e1r}), $\|U_n\|\rightarrow c_1>0$, and $\|F_n\|\sim c_2>0$ for some positive constants $c_1$ and $c_2$. Thus, the proof is complete by taking $\frac{F_n}{\|F_n\|}$ instead of $F_n$.
\end{proof}

\section{Delay at the parabolic part (system (\ref{s1}))}

In this section we prove essentially that when the delay is present at the parabolic part, the corresponding semigroup is also  non differentiable.

\subsection{Lack of differentiability of semigroup associated to (\ref{s1})}

\begin{theorem}\label{thch2}
Even if $a>\kappa$ and $\xi\in \mathring{\mathcal{J}}_{a,\kappa,\tau}$, the semigroup $e^{t\mathcal{A}_{\alpha,\beta}}$ is not differentiable, in the region $Q$.
\end{theorem}

\begin{proof}
In view of Remarks \ref{rms} and Lemma \ref{cnc} it suffices to show that there exists a sequence of positive real numbers $\lambda_n$ and for each $n$, an element $F_n \in \mathcal{H}$  such that $\lim_{n\rightarrow\infty}\lambda_n=\infty,\; ||F_n||=1$ and the solution $U_n$ of
\begin{equation}\label{e2}
(i\lambda_n I - \mathcal{A}_{\beta,\alpha})U_n=F_n
\end{equation}
satisfies
$$\lim_{n\rightarrow\infty}||U_n||=c>0.$$
 

For each $n\in\mathbb{N}$, we choose $F_n$ of the form $F_n=(0,0,0,h_n)^\top$ where $h_n \in L^{2}\big((0,1),H\big) $ will be chosen later.

The solution $U_n=(u_n,v_n,\theta_n,z_n)^\top$ of (\ref{e1}) should satisfies 
\begin{equation}\label{s62"}
\ds\mathbf{i}\lambda_n A^{\frac{1}{2}}u_n-A^{\frac{1}{2}}v_n=0 \textrm{ in } H,
\end{equation}
\begin{equation}\label{s63"}
\ds\mathbf{i}\lambda_n v_n+A^{\frac{1}{2}}\big(A^{\frac{1}{2}}u_n-A^{\beta-\frac{1}{2}}\theta_n\big)=0 \textrm{ in } H,
\end{equation}
\begin{equation}\label{s64"}
\ds\mathbf{i}\lambda_n\theta_n+\kappa A^{\frac{\alpha}{2}}
\big(z_n(1)+\frac{a}{\kappa}A^{\frac{\alpha}{2}}\theta_n
+\frac{1}{\kappa} A^{\beta-\frac{\alpha}{2}}v_n\big)=0 \textrm{ in } H,
\end{equation}
\begin{equation}\label{s65"}
\ds\mathbf{i}\lambda_nz_n+\frac{1}{\tau}\partial_{\rho}z_n
h_n=0 \textrm{ in }L^2\big((0,1),H\big).
\end{equation}

By integration of the identity (\ref{s65"}), we obtain
\begin{equation}\label{s65""}
\ds z_n(\rho)=A^{\frac{\alpha}{2}}\theta_n e^{-\mathbf{i}\lambda_n\tau\rho}+
\tau e^{-\mathbf{i}\tau\lambda_n\rho}\int_{0}^{\rho}e^{\mathbf{i}\tau\lambda_ns}h_n(s)ds.
\end{equation}
From (\ref{s64"}) we deduce that
\begin{equation}\label{s65"""}
\kappa z_n(1)=-\left( \mathbf{i}\lambda_nA^{-\frac{\alpha}{2}}\theta_n+aA^{\frac{\alpha}{2}}\theta_n
+A^{\beta-\frac{\alpha}{2}}v_n\right).
\end{equation}
Taking $\rho=1$ in (\ref{s65""}) and comparing the obtained result to (\ref{s65"""}) to get
\begin{equation}\label{s65bis}
\kappa\tau e^{-\mathbf{i}\lambda_n\tau}\int_{0}^{1}e^{\mathbf{i}\tau\lambda_ns}h_n(s)ds=\Phi_n.
\end{equation}
where
\begin{equation}\label{phi"}
\Phi_n=-\left( \mathbf{i}\lambda_nA^{-\frac{\alpha}{2}}\theta_n+\left( a+\kappa e^{-\mathbf{i}\tau\lambda_n}\right) A^{\frac{\alpha}{2}}\theta_n
+A^{\beta-\frac{\alpha}{2}}v_n\right).
\end{equation}

We choose $$h_n(s):=\frac{1}{\kappa\tau}e^{\tau\mathbf{i}\lambda_n(1-s)}\Phi_n$$
which guaranties the equality (\ref{s65"""}) and gives
 \begin{equation}\label{s31bis}
\ds z_n(\rho)=A^{\frac{\alpha}{2}}\theta_n  e^{-\mathbf{i}\lambda_n\tau\rho}+
\frac{\rho}{\kappa} e^{\tau\mathbf{i}\lambda_n(1-\rho)}\Phi_n.
\end{equation}
For each $n\in \mathbb{N}$, let $e_n$ an eigenvector of $A$ with $\|e_n\|=1$ and  $Ae_n=\mu_ne_n$, where $\mu_n$ is a sequence of real numbers such that $\lim\limits_{n\rightarrow\infty}\mu_n=\infty$. We take  $$v_n=\frac{1}{\mu_n^p}e_n,\;\;\lambda_n=\mu_n^q$$
where $p$ and $q$ are real positive numbers to be specified later,  such that $U_n$ and $F_n$ fulfill the desired conditions.

It follows from (\ref{s62"})-(\ref{s64"}) and (\ref{phi"}) that:
\begin{eqnarray}
A^{\frac{1}{2}}u_n&=&-i\mu_n^{\frac{1}{2}-p-q}e_n,\label{3.12}\\
\theta_n&=&i\left(\mu_n^{-\beta-p+q}-\mu_n^{1-\beta-p-q}\right)e_n,
\end{eqnarray}
and
\begin{eqnarray}
\Phi_n&=&\left[\mu_n^{-\frac{\alpha}{2}-\beta-p+2q}-\mu_n^{1-\frac{\alpha}{2}-\beta-p}-\mathbf{i}\left( a+\kappa e^{-\tau\mathbf{i}\lambda_n}\right)\left(\mu_n^{\frac{\alpha}{2}-\beta-p+q}-\mu_n^{1+\frac{\alpha}{2}-\beta-p-q}\right)\right.\notag\\
&-&\left. \mu_n^{\beta-\frac{\alpha}{2}-p}\right] e_n.\label{3.15}
\end{eqnarray}

Case 1: $\frac{\alpha}{2}+\frac{1}{2}-\beta> 0$ and $\alpha>0$. Choosing $p=\frac{\alpha}{2}-\beta+1-\delta$ and $q=\delta$, where $\delta>0$, small with  $\delta<\alpha,\frac{1}{2},\frac{1+\alpha}{3}$ and $\delta<1+\alpha-2\beta$,  then (\ref{3.12})-(\ref{3.15}) give
\begin{eqnarray*} 
v_n&=&o(1),\\
A^{\frac{1}{2}}u_n&=&\mu_n^{-\frac{1}{2}-\frac{\alpha}{2}+\beta}e_n=o(1),\\
\theta_n&=&\mu_n^{-\frac{\alpha}{2}-1+2\delta}e_n=o(1),\\
\end{eqnarray*}
\begin{eqnarray*}
\Phi_n&=&\left[\mu_n^{-\alpha-1+3\delta}-\mu_n^{\delta-\alpha}-i\left( a+\kappa e^{-\tau\mathbf{i}\lambda_n}\right)\left(\mu_n^{-1+2\delta}-1\right)
- \mu_n^{\delta-(\alpha+1-2\beta)}\right] e_n\\
&=&\left[i\left( a+\kappa e^{-\tau\mathbf{i}\lambda_n}\right)+o(1)\right]e_n.
\end{eqnarray*}
One checks that $h_n\in L^2\big((0,1),H\big) $, $\|U_n\|\rightarrow c_1>0$, and $\|F_n\|\sim c_2>0$ for some positive constants $c_1$ and $c_2$.

Finally, recall that $U_n$ is solution of (\ref{e1}). Thus, the proof is complete by choosing $\frac{F_n}{\|F_n\|}$ instead of $F_n$.

Case 2: If $\frac{\alpha}{2}+\frac{1}{2}-\beta= 0$ we choose $p=q=\frac{1}{2}$, and if $\alpha=0$ and $\beta<\frac{1}{2}$ we choose $p=1-\beta$ and $q=\delta$ small. Then we obtain the same result as in the first case.
\end{proof}
\subsection{Some related systems}
Suppose that $A^{\alpha}=BB^*$
where $B:D(B) \subset H\rightarrow H$ is a closed densely defined linear unbounded operator. The 
 assumption 
$$B^*\theta(t-\tau)=g_0(t),  \quad t\in(0,\tau)$$
is given  instead of the assumption
$$A^{\frac{\alpha}{2}}\theta(t-\tau)=g_0(t),  \quad t\in(0,\tau).$$
 
Precisely, we consider the following system
 \begin{equation}\label{s1''}
\left\{
\begin{array}{ll}
u''(t)+ Au(t)+A^{\beta}\theta(t)=0, & \quad t\in(0,+\infty), \\
\theta'(t)+\kappa A^\alpha \theta(t-\tau)+aA^\alpha \theta(t)-A^{\beta}u'(t)=0, & \quad t\in(0,+\infty), \\
u(0)=u_{0}, u'(0)=u_{1}, \theta(0)=\theta_{0}, &  \\
\ds B^{*}\theta(t-\tau)=\psi(t-\tau), & \quad t\in(0,\tau).
\end{array}
\right.
\end{equation}
 
 We suppose that  $\beta \leq \frac{1}{2}$ and $$D(B^*)= D(A^{\frac{\alpha}{2}})$$ and $$\|(A^{\frac{\alpha}{2}})v\|_H\leq c\|B^*v\|_{H},\;\;\forall\
v\in D(B^{*})=D(A^{\frac{\alpha}{2}}).$$

Here we take
$$\ds z(\rho,t):=B^{*}\theta(t-\tau\rho),\quad\rho\in(0,1), t>0.$$

Define
$U=(u,u',\theta,z)^{\top}$,
then problem (\ref{s1''}) can be formulated as a first order system of the form

\begin{equation*}
\left\{
\begin{array}{ll}
U'=\mathcal{A}_{\alpha,\beta}U, & \\
\ds U(0)=\big(u_{0},u_{1},\theta_{0},\psi(-\tau \cdot)\big)^{\top},
\end{array}
\right.
\end{equation*}
where the operator $\mathcal{A}_{\alpha,\beta}$ is defined by
$$\mathcal{A}_{\alpha,\beta}
\left(
\begin{array}{c}
u \\
v \\
\theta \\
z \\
\end{array}
\right)=
\left(
\begin{array}{c}
v \\
-A^{\frac{1}{2}}\big( A^{\frac{1}{2}}u-A^{\beta-\frac{1}{2}}\theta\big)  \\
\ds-\kappa B\big(z(1)+\frac{a}
{\kappa}B^*\theta\big)- A^{\beta}v \\
\ds-\frac{1}{\tau}z_{\rho} \\
  \end{array}
\right),$$
with domain
$$D(\mathcal{A}_{\alpha,\beta})=\left\{
\begin{array}{c}
(u,v,\theta,z)^{\top}\in D(A^{\frac{1}{2}})\times D(A^{\frac{1}{2}})
\times D(B^{*})\times H^{1}\big((0,1),H\big):\; z(0)=B^*\theta, \\
\ds   A^{\frac{1}{2}}u-A^{\beta-\frac{1}{2}}\theta\in D(A^{\frac{1}{2}})  \quad\textrm{ and }\quad z(1)+\frac{a}
{\kappa}B^*\theta\in D(B)
\end{array}
\right\}.$$ 

in the Hilbert space
$$\mathcal{H}=D(A^{\frac{1}{2}})\times H\times H\times L^{2}\big((0,1),H\big),$$
equipped with the scalar product
$$\ds\big((u,v,\theta,z)^\top,(u_1,v_1,\theta_1,z_1)^\top\big)_{\mathcal{H}}
=\big(A^{\frac{1}{2}}u,A^{\frac{1}{2}}u_1\big)_{H}
+(v,v_1)_{H}+(\theta,\theta_1)_{H}+\xi\int_{0}^{1}(z,z_1)_{H}d\rho,$$
where $\xi>0$ is a positive constant.

\medskip

We know that For $\xi \in \mathcal{J}_{\alpha,\kappa,\tau}$ and $a\geq \kappa$, the operator $\mathcal{A}_{\alpha,\beta}$ generates a $\mathcal{C}_0$-semigroup of  contractions, which is at least polynomially stable in the region $Q$ for $a>\kappa$ and $\xi\in \mathring{\mathcal{J}}_{a,\kappa,\tau}$ \cite{ASS24}.

We can prove as in the previous case that:  
\begin{theorem}\label{thch1}
The $\mathcal{C}_0$-semigroup $e^{\mathcal{A}_{\alpha,\beta}t}$ is  not differentiable in the region $Q$ even for $\xi\in \mathring{\mathcal{J}}_{a,\kappa,\tau}$, $a>\kappa$.
\end{theorem}
\begin{exam} 
We consider the one dimensional delayed thermoelastic  beam model in $(0,L)$, $L>0$:
\begin{equation}\label{exp33}
\left\{
\begin{array}{ll}
u_{tt}(t,x)+u_{xxxx}(x,t)-\theta_{xx}(x,t)=0, &
\quad (x,t)\in(0,L)\times(0,+\infty), \\
\theta_{t}(x,t)-\kappa\theta_{xx}(x,t-\tau)-a\theta_{xx}(x,t)+u_{txx}(x,t)=0, &
\quad(x,t)\in(0,L)\times(0,+\infty), \\
u(0,t)=u(L,t)=u_{xx}(0,t)=u_{xx}(L,t)=0, & \quad t\in(0,+\infty), \\
u(x,0)=u^0(x),u_t(x,0)=u^1(x), & \quad t\in(0,\tau), \\
\theta(0,t)=\theta(L,t)=0, & \quad t\in(0,+\infty), \\
\theta(x,0)=\theta^0(x), &  \\
\theta_x(x,t)=g_0(x,t), &\quad-\tau\leq t<0, x\in (0,L),
\end{array}
\right.
\end{equation}
where $L$ and $a$ are real positive constants.

\medskip

It is a concrete example of the related system, with $H=L^2(0,L)$, $A=\frac{\partial^4}{\partial x^4}:D(A)=H_0^2(0,L)\cap H^4(0,L)\rightarrow L^2(0,L)$, $\alpha=\beta=\frac{1}{2}$,   $A^\alpha=A^\beta=A^{\frac{1}{2}}=-\frac{\partial^2}{\partial x^2}:D(A^{\frac{1}{2}})=H_0^1(0,L)\cap H^2(0,L)\rightarrow L^2(0,L)$,  $B=-\frac{\partial}{\partial x}:D(B)=H^1(0,L)\rightarrow L^2(0,L)$, $B^*=\frac{\partial}{\partial x}:D(B^*)=H_0^1(0,L)\rightarrow L^2(0,L)$. We have that  $A^{\alpha}=A^{\frac{1}{2}}=BB^*=$, $D(A^{1/4})=D(B^*)=H_0^1(0,L)$ and $\|(A^{1/4})v\|_H=\|B^*v\|_{H},\;\;\forall\
v\in D(B^{*})=H_0^1(0,L)$.    

\medskip

To simplify writing, denote $V=H_{0}^{1}\big(0,L\big) \cap H^{2}\big(0,L\big)$.  The operator $\mathcal{A}_{\frac{1}{2},\frac{1}{2}}$ is given by
\begin{equation*}
\ds\mathcal{A}_{\frac{1}{2},\frac{1}{2}}
\left(
\begin{array}{c}
u \\
v \\
\theta \\
z \\
\end{array}
\right)=
\left(
\begin{array}{c}
v \\
\ds\left( -u_{xx}+\theta\right)_{xx}  \\
\kappa\big(z(1)+\frac{a}{\kappa}\theta_{x}\big)_x-v_{xx} \\
\ds-\frac{1}{\tau}z_{\rho} \\
  \end{array}
\right),
\end{equation*}
with domain
\begin{equation*}
D(\mathcal{A}_{\frac{1}{2},\frac{1}{2}})=\left\{
\begin{array}{c}
(u,v,\theta,z)^{\top}\in V\times V 
\times  H_{0}^{1}\big(0,L\big) \times H^{1}\big((0,1),L^{2}(0,L)\big): \\
\ds \quad z(0)=\theta_x,\quad -u_{xx}+\theta\in V \quad\textrm{ and }\quad z(1)+\frac{a}{\kappa}\theta_x\in H^{1}\big(0,L\big)
\end{array}
\right\},
\end{equation*}
in the Hilbert space
\begin{equation*}
\mathcal{H}= V \times
L^{2}\big(0,L\big) \times L^{2}\big(0,L\big)\times
H^{1}\big((0,1)\times L^{2}(0,L)\big),
\end{equation*}
In view of Theorem \ref{thch1} and even if  $a> \kappa$ and $\xi \in \mathring{\mathcal{J}}_{a,\kappa,\tau}$ the associated semigroup $e^{\mathcal{A}_{\frac{1}{2},\frac{1}{2}}t}$ is not differentiable.
\end{exam}
\begin{remark}
We can consider the thermoelastc string with delay at the heat \cite{MuKa13}. We find that the associated semigroup (in the appropriate Hilbert space) is not differentiable.
\end{remark}
\begin{remark}
A small modification in the proof of Theorem \ref{thch2} allows us to conclude that the semigroup associated to system (\ref{s1}) or (\ref{s1''}) is not exponentially stable in the Hilbert space $\mathcal{H}$ for $(\beta,\alpha)$ in the region $S_1\cup S_2$.

\end{remark}

\end{document}